\documentclass[12pt,leqeq]{article}
\usepackage{amssymb,amsthm,mathrsfs}
\usepackage{amsmath}
\usepackage{amscd}
\usepackage{xy}
\xyoption{all}
\usepackage[dvips]{graphicx}
\usepackage{graphpap}
\usepackage{pstricks}
\usepackage{pst-all}
\usepackage{pstcol}
\usepackage{color}
\usepackage{braket}
\newtheorem{thm}{Theorem}
\newtheorem{lem}[thm]{Lemma}
\newtheorem{cor}[thm]{Corollary}
\newtheorem{prop}[thm]{Proposition} 
\newtheorem{rem}[thm]{Remark}

\newtheorem{examp}[thm]{Example}

\date{}
\begin{document}
\setlength{\baselineskip}{16pt}
\title{Continuous Analogues for the Pochhammer Symbol}
\author{Rafael D\'\i az}
 \maketitle
\begin{abstract}
We develop continuous analogues for the Pochhammer symbol following the line of research initiated by D\'iaz and Cano on the construction of continuous analogues for combinatorial objects.
\end{abstract}

\section{Introduction}

Let $R$ be a semiring, the Pochhammer symbol  $\ r(x,y,n) \ $ is given by
$$r(x,y,n)\ = \ \prod_{l=0}^{n-1}(x+ly) \ = \ x(x+y)(x+2y)\cdots (x+(n-1)y).$$
It defines a function $\ r: R  \times R \times \mathbb{N}  \rightarrow R \ $ sending
$\ (x,y,n) \ $ to $\ r(x,y,n), \ $ where by convention  $\ r(x,y,0)=1. \ $ Among the many prominent usages in the literature of the symbol $\ r(x,y,n)\ $  we single out just a few:

\begin{description}
\item[a)] \underline{Factorial numbers:} \ \ \ $n!=r(n,-1,n)=r(1,1,n)=1\cdot 2\cdots (n-1) \cdot n.$
\item[b)] \underline{Rising factorials:}\ \ \ $(x)_{\overline{n}}=r(x,1,n)= x(x+1)(x+2)\cdots (x+(n-1)).$
\item[c)] \underline{Lowering factorials:}\ \ \ $(x)_{\underline{n}}=r(x,-1,n)= x(x-1)(x-2)\cdots (x-(n-1)).$
\item[d)] \underline{Double factorials:} \ \ \ $(2n-1)!!= r(1,2,n)=r(2n-1,-2,n-1)=(2n-1)\cdots 3\cdot 1.$
\item[e)] \underline{Pochhammer $k$-symbol:}\ \ \ $(x)_{n,k}=r(x,k,n)=x(x+k)(x+2k)\cdots (x+(n-1)k).$
\end{description}

Our goal in this work is to introduce  continuous analogues for the function
$\ r(x,y,n), \ $ specially in regard to the third variable $\ n\in \mathbb{N}. \ $ Doing this demands further properties from the semiring $R,$ thus we work over the reals $\ \mathbb{R}. \ $ A continuous extension of  $\ r(x,y,n)\ $  from the domain   $\ \mathbb{R}_{>0} \times\mathbb{ R }_{>0}\times \mathbb{N} \  $ to the domain $(\mathbb{C}\setminus \mathbb{R}_{\leq 0}) \times\mathbb{ R }_{>0}\times \mathbb{R}_{\geq 0}\ $ is given by:
\begin{equation}\label{30114a}
r(x,y,z) \ = \    \dfrac{\Gamma_y(x+yz)}{\Gamma_y(x)} \ = \ y^z\frac{\Gamma(\frac{x}{y}+z)}{\Gamma(\frac{x}{y})}\ = \ y^zr(\frac{x}{y},1,z),
\end{equation}
where \cite{dp1} the $y$-Gamma function $\ \Gamma_y \ $  for $\ y>0 \ $ is given by
$$\Gamma_y(x) \ = \ \lim_{n\rightarrow \infty} \ \dfrac{n!\ y^n  (ny)^{\frac{x}{y}-1}}{r(x,y,n)}
\ \ \ \ \mbox{for} \ \ \ x\in  \mathbb{C}\setminus y\mathbb{Z}_{\leq 0}, \ \ \ \ \ \
\ \ \ \ \ \ \ \ \ $$
$$\Gamma_y(x) \ = \ \int_{0}^{\infty}t^{x-1}e^{\frac{-t^y}{y}}dt \ \ \ \ \mbox{for} \ \ \
\mathrm{Re}(x)>0.\ \ \ \ \ \ \ \ \ \ \ \ \ \ \ \ \ \ \ \ \ \ \ \    $$

We actually study a couple of functions $\ \widetilde{r}(x,y,n) \ $ and $\ \rho(x,y,z)\ $ that may be both regarded as continuous analogues of the function $\ r(x,y,n), \ $ in the sense that they solve continuous "counting" problems analogous to the discrete counting problem solved by $\ r(x,y,n); \ $ this work follows the line of research opened in \cite{cd1, cd2} and further developed by Cano, Carrillo \cite{cc}, Le, Robins, Vignat, Wakhare \cite{lrvw},  O'Dwyer \cite{o}, and Vignat, Wakhare \cite{vw}. In a nutshell, the idea is to construct continuous analogues for combinatorial objects by replacing the operation of counting weighted lattice points in top dimensional convex polytopes, by the operation of computing integrals on convex polytopes. Since the cardinality of many combinatorial objects may be represented as counting weighted lattice points in convex polytopes, this methodology is expected to have a wide range of applications. The continuous analogue obtained depends on the choice of such a representation.

The relationship between the discrete and the continuous quantities built with the methodology above is actually quite subtle and complex. Consider the following simple but illustrative example.
For $\ x>0 \ $ let $\ \Box^2_x\subseteq \mathbb{R}^2 \ $ be the square with vertices $\ (0,0),(0,x),(x,0),(x,x).\ $ We have that $\ \mathrm{area}(\Box^2_x)=x^2\ $ and $\ |\mathbb{Z}^2\cap \Box^2_x|=\lceil x \rceil^2, \ $ thus $$\ \  \displaystyle \lim_{x\to \infty}\frac{|\mathbb{Z}^2\cap \Box^2_x|}{\mathrm{area}(\Box^2_x)} =1 \ \ \ \ \mbox{and} \ \ \ \ \displaystyle \lim_{x\to 0}\frac{|\mathbb{Z}^2\cap \Box^2_x|}{\mathrm{area}(\Box^2_x)} =\infty . \ \ $$ So in this example replacing "counting lattice points" by "computing volumes" is not a drastic change for large volume polytopes, but may actually be quite significative for small volume polytopes.  This line of thinking goes back to Ehrhart \cite{ehrhart}; for further information the reader may consult \cite{bar, blve,dr,lrvw} and the references therein.

For our current purposes, we must consider no just a single polytope, but rather sums over  finite o even infinite families of polytopes;  moreover we are forced to consider polytopes of various sizes, so no straightforward a priori relation between the discrete and its continuous analogues should be expected.

This work is organized as follows. In Section 2\  we describe the Pochhammer symbol in terms of counting weighted lattice points in convex polytopes; in Section 3 we introduced a first continuous analogue for the Pochhammer symbol proceeding first by analogy and then by extension; in Section 4 we introduce a second continuous analogue for the Pochhammer proceeding twice by analogy.

\section{Pochhammer Symbol and Convex Polytopes}

We proceed to describe the Pochhammer symbol $\ r(x,y,n)\ $ in terms of counting weighted lattice points in convex polytopes. For $\ n\geq 1\ $ let $\ \mathrm{P}_k[0,n-1]\ $ be the set of subsets of $\ \{0,...,n-1\}\ $ with $\ k \ $ elements. We have that:
$$r(x,y,n)\ = \ \prod_{l=0}^{n-1}(x+ly) \ = \ \sum_{A\subseteq [0,n-1]}r_S\ x^{|A|}y^{n-|A|}
\ = \ \sum_{k=0}^{n} \bigg(\sum_{A\in \mathrm{P}_k[0,n-1]}r_A\bigg) x^{k}y^{n-k},$$
$$\mbox{where} \ \ \ r_A=  \prod_{i\in A^c}i \ \ \
\mbox{for}  \ \ \ A\subseteq [0,n-1]. $$ Therefore the following well-known result holds.

\begin{prop}\label{r2}{\em Let $ \ (x,y,n) \in  \mathbb{R} \times\mathbb{ R }\times \mathbb{N}, \ $ then
$\ \ \displaystyle r(x,y,n) =  \sum_{k=0}^{n} r_{n,k}\ x^{k}y^{n-k},$
$$ \mbox{where} \ \ r_{0,0}=1, \ \  \text{and for} \ \ n\geq k>0 \ \ \text{we have that} \ \ r_{n,0}=0, \  \
\mbox{and} $$
$$r_{n,k}\ =\  \sum_{ (s_1,...,s_{n-k})  \in  \mathbb{Z}^{n-k} \cap \Delta^{n-k,\ast}_{n-1} } s_1\cdots s_{n-k}  \ \ \ \ \  \ \mbox{with}$$
 $$\ \ \displaystyle \Delta^{n-k,\ast}_{n-1} \ = \ \{(s_1,...,s_{n-k}) \in \mathbb{R}^{n-k} \ | \
0\leq s_1 < \cdots < s_{n-k} \leq n-1 \}.\ $$
}
\end{prop}

\begin{examp}\label{r3}{\em \ \\
a) $n!=r(1,1,n)=\sum_{k=0}^{n} r_{n,k}. $\\
b)  $(x)_{\overline{n}}=r(x,1,n)= \sum_{k=0}^{n} r_{n,k}x^{k}. \ \ \ $ \\
c) $(x)_{\underline{n}}=r(x,-1,n)=\sum_{k=0}^{n} (-1)^{n-k}r_{n,k} x^{k}=\sum_{k=0}^{n} s_{n,k} x^{k}.$ \\
Furthermore $x^n=\sum_{k=0}^{n} S_{n,k} r(x,-1,k).$\\
d) $(2n-1)!!=r(1,2,n)=\sum_{k=0}^{n} r_{n,k} 2^{n-k}.$
}
\end{examp}

The coefficients $\ r_{n,k}\ $ are the unsigned Stirling numbers of the first kind, they count the number of permutations of an $n$-element set with $k$ cycles. The Stirling numbers of the first kind are given by $\ s_{n,k}=(-1)^{n-k}r_{n,k};\ $ the Stirling numbers of the second kind $\ S_{n,k} \ $ are defined via the identities
$\  \sum_{k\leq l \leq n}s_{n,l}S_{l,k}  =  \delta_{n,k}  =   \sum_{k\leq l \leq n}S_{n,l}s_{l,k};\ $ the number
 $\ |S_{n,k}| \ $ counts partitions of a $n$-element set with $k$ blocks.

Note that
$\ \  \Delta^{n-k,\ast}_{n-1} \ \subseteq \ \Delta^{n-k}_{n-1} \ = \ \{(s_1,...,s_{n-k}) \in \mathbb{R}^{n-k} \ | \
0\leq s_1 \leq \cdots \leq s_{n-k} \leq n-1 \}.\ \ $  By Proposition \ref{r3} \ the Pochhammer symbol $\ r(x,y,n) \ $ is a finite sum where each summand is itself a weighted  sum over the lattice points on the convex polytope $\ \Delta^{n-k}_{n-1} , \ $ excluding some boundary points.
Integrals of continuous functions over $\ \Delta^{n-k}_{n-1} \ $ do not change if some boundary components are removed, thus to introduce our continuous analogues we work with $\ \Delta^{n-k}_{n-1}\ $ instead of $\ \Delta^{n-k,\ast}_{n-1}.$\\

Next result is key for the rest of this work. For $\ k\in \mathbb{N}_{>0}\ $ and  $\ x\in \mathbb{R}_{\geq 0} \ $ set
$$\ \Delta^k_x=\{(s_1,...,s_{k}) \in \mathbb{R}^{k} \ | \
0\leq s_1 \leq \cdots \leq s_{k} \leq x \} \ \ \ \text{and}
\ \ \  \displaystyle   a_{x,k}=   \int_{\Delta^k_x} s_1\cdots s_k \ ds_1 \cdots ds_k.$$

\begin{lem}\label{r1}{\em
  For  $\ x\in \mathbb{R}_{\geq 0},\ k\in \mathbb{N},\ $ setting $\ \mathrm{vol}(\Delta^0_x)=1 \ $ and $\ a_{x,0}=1\ $ we have that:
$$\mathrm{vol}(\Delta^k_x)  =   \dfrac{x^k}{k!} \ \ \ \ \ \ \ \  \mbox{and} \ \ \ \ \ \ \ \
a_{x,k} =  \mathrm{vol}(\Delta^k_{\frac{x^2}{2}})     =  \dfrac{x^{2k}}{2^k \ k!}  =
(2k-1)!!\frac{x^{2k}}{(2k)!}.$$

}
\end{lem}

\begin{proof}
The results follow from the recursions
  $$\mathrm{vol}(\Delta^k_x) \ = \ \int_{0}^{x}\int_{0}^{s_k}\cdots \int_{0}^{s_2}
ds_1 \cdots ds_k \ = \  \int_{0}^{x} \mathrm{vol}(\Delta^{k-1}_{s_k})\ ds_k \ ; \ \ \ \ \ \  $$
$$ a_{x,k}\ = \ \int_{0}^{x}\int_{0}^{s_k}\cdots \int_{0}^{s_2}
s_1\cdots s_k \ ds_1 \cdots ds_k\ = \
\int_{0}^{x} s_k\ a_{s_k,k-1}\ ds_k  \ .$$
\end{proof}

\section{Continuous Analogue - First Approach}

We are ready to introduce our first continuous analogue for the Pochhammer symbol; in this approach we proceed first by analogy and then by extension. In accordance with the methodology
outlined in the introduction, Proposition \ref{r2}, and Lemma \ref{r1},   we replace the numbers $r_{n,k}$ by their continuous analogues $\ \widetilde{r}_{n,k} \ $  given for $\ n\geq k>0\ $ by
$\ \widetilde{r}_{0,0}=1, \ $ $\ \widetilde{r}_{0,k}=0, \ $ and
$$ \widetilde{r}_{n,k}=a_{n-1,n-k}\ = \ \int_{0}^{n-1}\int_{0}^{s_{n-k}}\cdots \int_{0}^{s_2}
s_1\cdots s_{n-k} \ ds_1 \cdots ds_{n-k}\ = \ \dfrac{(n-1)^{2(n-k)}}{2^{n-k}(n-k)!} \ .$$
 Therefore we introduce the function $\ \widetilde{r}:\mathbb{R} \times\mathbb{ R }\times \mathbb{N}  \rightarrow \mathbb{R}\ $ given by
\begin{equation}\label{e1}
\widetilde{r}(x,y,n)= \sum_{k=0}^{n} \widetilde{r}_{n,k}\ x^{k}y^{n-k}
\end{equation}
as a continuous analogue for the Pochhammer symbol.  So $\ \widetilde{s}_{n,k}=(-1)^{n-k} \widetilde{r}_{n,k} \ $ give analogues for the Stirling numbers of the first kind. Analogues $\ \widetilde{S}_{n,k} \ $ for the Stirling numbers of the second kind  are defined by the identities $$\sum_{k\leq l \leq n}\widetilde{s}_{n,l}\widetilde{S}_{l,k} \ = \ \delta_{n,k} \ = \  \sum_{k\leq l \leq n}\widetilde{S}_{n,l}\widetilde{s}_{l,k}.$$

Since $\ \widetilde{r}_{n,k}, \ \widetilde{s}_{n,k}, \ \widetilde{S}_{n,k} \ $ are rational numbers, we search for a combinatorial interpretation of them in terms of finite groupoids \cite{bd,d}.

A finite groupoid is a category with a finite set of morphisms, all of them invertible.
The cardinality $\ |G| \ $ of a finite groupoid $\ G \ $ is the rational number given by
$$|G| =  \sum_{\overline{a}\in \mathrm{ob}(G)_{\simeq}}\frac{1}{|G(a,a)|},$$
where $\ \mathrm{ob}(G)_{\simeq} \ $ is the quotient of the set $\ \mathrm{ob}(G) \ $ of objects in $\ G \ $ by the relation $\ \simeq, \ $ where two objects are related if and only if they are isomorphic; $\ G(a,a) \ $ is the set of morphism from object $a$ to itself, and $\ |G(a,a)|\ $ its cardinality.  If group $\ G\ $ acts on the finite set $\ X, \ $ we let $\ X_G\ $ be the groupoid such that $\ \mathrm{ob}(X_G)=X,\ $ and the set $\ X_G(a,b)\ $ of morphisms from $\ a \ $ to $\ b \ $ in  $\ X_G\ $ is given by $\ X_G(a,b)= \{g\in G \ | \ ga=b \}. \ $ Composition of morphisms in  $\ X_G \ $ is given by the product on $\ G. \ $  We shall use the well-known  fact $ \  \displaystyle |X_G| =  \frac{|X|}{|G|}. \ $ \\

Let us introduce finite groupoids $\ G_{n,k}, \ G^e_{n,k}, \ G^o_{n,k}.  \ $ Set $\ [n-k]=\{1,...,n-k\}, \ $ $\ -[n-k]=\{-(n-k),...,-1\}, \ $
$\ \mathbb{Z}_2 = \{-1, 1\}, \ $ and $$\ S_{n-k}= \{\sigma:[n-k]\to [n-k] \ | \ \sigma \  \text{bijection}\}. \ $$ Let $\ \mathbb{Z}_2^{n-k}\rtimes S_{n-k} \ $  be the group
with the semidirect product  $$ \ (a,\sigma)(b,\tau)=(a\sigma(b),\sigma\tau) \ \ \text{with} \ \ \sigma(b)_i=b_{\sigma^{-1}(i)}. \ $$ The group $\ \mathbb{Z}_2^{n-k}\rtimes S_{n-k} \ $
acts on $\ -[n-k] \sqcup [n-k]\ $ via $\ (a,\sigma)(\pm i)=\pm a_{\sigma(i)}\sigma(i) \ $ for $\ (a,\sigma)\in \mathbb{Z}_2^{n-k}\rtimes S_{n-k}, \ i \in [n-k]. \ $
We let $\ G_{n,k} \ $ for $\ n> k >0 \ $ be the groupoid given by
$$ G_{n,k} \ = \ \mathrm{Map}(-[n-k]\sqcup [n-k], [n-1])_{\mathbb{Z}_2^{n-k}\rtimes S_{n-k}} ,$$
i.e. an object in $\  G_{n,k} \ $ is a map $\ f:-[n-k]\sqcup [n-k] \to [n-1], \ $ and a morphism
in $\ G_{n,k} \ $ from $f$ to $g$ is an element $(a, \sigma) \in \mathbb{Z}_2^{n-k}\rtimes S_{n-k}$ such that $\ (a, \sigma)f=g, \ $ where $$
\big((a, \sigma)f\big)(\pm i)=f\big((a, \sigma)^{-1}(\pm i)\big)=f\big((\sigma^{-1}(a),\sigma^{-1})(\pm i)\big)=
f(\pm a_i\sigma^{-1}(i))\ \  \text{for} \ \ i \in [n-k]. $$ By definition we have that $$|G_{n,k}|=\frac{(n-1)^{2(n-k)}}{2^{n-k}(n-k)!}.  $$
Let the gruopoid $\ G^e_{n,k} \ $ for $\ n\geq k >0 \ $ be given by $\ G^e_{n,k}=(X^e_{n,k})_{\mathbb{Z}_2^{n-k}\rtimes S_{n-k}}, \ $ where an element of $\ X^e_{n,k} \ $ is an ordered tuple $\ (A_1,...,A_l;f) \ $ such that $\ l\geq 1 \ $ is an even number,  $\ \{A_1,...,A_l\} \ $ is a partition of $\ [n-k], \ $ and
$\ f: -[n-k] \sqcup [n-k] \to [n-1] \ $ is a map such that $$f(\pm s) \ \ \ \text{belongs to} \ \ \ \coprod_{j\leq i}A_j \ \sqcup \ [n-k+1,n-1]   \ \ \ \text{for all} \ \ \
s\in A_i.$$
For $\ \sigma\in S_{n-k} \ $ we let $\ \widehat{\sigma} \in S_{n-1} \ $ be the permutation given by $\ \widehat{\sigma}(i)=\sigma(i) \ $ if $\ i\in[n-k], \ $ and
$\ \widehat{\sigma}(i)=i \ $ if $\ i\in[n-k+1,n-1]. \ $ The group
$\ \mathbb{Z}_2^{n-k}\rtimes S_{n-k}\ $ acts on $\ X^e_{n,k}\ $ as follows
$$\ (a,\sigma)(A_1,...,A_l;f) \ = \ (\sigma(A_1),...,\sigma(A_l); \widehat{\sigma} f(a,\sigma)^{-1}), \ \ \ \text{where}$$
$$\  \widehat{\sigma} f(a,\sigma)^{-1}(\pm i)=\widehat{\sigma}
\big(f(\pm a_i \sigma^{-1}(i)) \big). \ $$ By definition we have that
$$|G^e_{n,k}|=\frac{1}{2^{n-k}(n-k)!}\sum^{n-k}_{\substack{ l=1 \\
l \ \mathrm{even} }}\sum_{\substack{ a_1+ \cdots +a_{l}=n-k \\
a_i>0 }}\binom{n-k}{a_1,...,a_l}
(a_1+k-1)^{2a_1} \cdots (a_1+\cdots +a_l+k-1)^{2a_l}.$$

The definition of $\ G^o_{n,k} \ $ follows the same patter as the definition of $\ G^e_{n,k}; \ $
the only change is that in   $\ G^o_{n,k} \ $ we only include tuples $\ (A_1,...,A_l;f)\ $ with
$\ l \ $ an odd number. It follows that
$$|G^o_{n,k}|=\frac{1}{2^{n-k}(n-k)!}\sum^{n-k}_{\substack{ l=1 \\
l \ \mathrm{odd} }}\sum_{\substack{ a_1+ \cdots +a_{l}=n-k \\
a_i>0 }}\binom{n-k}{a_1,...,a_l}
(a_1+k-1)^{2a_1} \cdots (a_1+\cdots +a_l+k-1)^{2a_l}.$$

\begin{thm}\label{2312a}{\em \
\begin{description}
\item[a)] For $ \ n\geq k >0, \ $ we have that  $\ \ \widetilde{s}_{n,k}\ = \ (-1)^{n-k}|G_{n,k}|.$
\item[b)] $\ \widetilde{S}_{0,0}=1, \  \widetilde{S}_{n,0} = 0,\ \widetilde{S}_{n,n} = 1, \ $ and for $ \ n> k >0 \ $   we have that
$$\widetilde{S}_{n,k}\ = \ (-1)^{n-k}\big(|G^e_{n,k}|-|G^o_{n,k}| \big).$$
\end{description}
}
\end{thm}

\begin{proof} Item a) is clear. Item b) follows by M$\ddot{\mathrm{o}}$bius inversion  \cite{d,r}:
$$\widetilde{S}_{n,k}\ = \ \sum_{l=1}^{n-k}(-1)^l\sum_{\substack{ i_0< \cdots <i_{l} \\
i_0=k,\  i_{l} =n }} \widetilde{s}_{i_l,i_{l-1}} \cdots \widetilde{s}_{i_1,i_{0}}\ =$$
$$ \sum_{l=1}^{n-k}(-1)^l\sum_{\substack{ i_0< \cdots <i_{l} \\
i_0=k,\  i_{l} =n }}
\frac{(-1)^{i_l-i_{l-1}}(i_l-1)^{2(i_l-i_{l-1})} \cdots (-1)^{i_1-i_{0}}(i_1-1)^{2(i_1-i_{0})}}{2^{i_l-i_{l-1}}(i_l-i_{l-1})! \cdots 2^{i_1-i_{0}}(i_1-i_{0})!}\ = $$
$$\frac{(-1)^{n-k}}{2^{n-k}(n-k)!} \sum_{l=1}^{n-k}(-1)^{l}\sum_{\substack{ i_0< \cdots <i_{l} \\
i_0=k,\  i_{l} =n }}\binom{n-k}{i_1-i_{0},...,i_l-i_{l-1}}
(i_1-1)^{2(i_1-i_{0}) }\cdots (i_l-1)^{2(i_l-i_{l-1})} \ = $$
$$ \frac{(-1)^{n-k}}{2^{n-k}(n-k)!} \sum_{l=1}^{n-k}(-1)^{l}\sum_{\substack{ a_1+ \cdots +a_{l}=n-k \\
a_i>0 }}\binom{n-k}{a_1,...,a_l}
(a_1+k-1)^{2a_1} \cdots (a_1+\cdots +a_l+k-1)^{2a_l}\ = $$
$$(-1)^{n-k}\big(|G^e_{n,k}|-|G^o_{n,k}| \big).$$
\end{proof}

\begin{cor}{\em For $ \ n> k >0,\ $ we have that
$$ \ |G^e_{n,k}|\ + \ |G^o_{n,k}| \ \leq \ \frac{(n-1)^{2(n-k)}S_{n-k}(n-k)}{2^{n-k}(n-k)!}, \ $$
where $\ S_k(n) =  1^k + \cdots + n^k. \ $
}
\end{cor}

\begin{proof}
$ \ \ \ \displaystyle |G^e_{n,k}|\ + \ |G^o_{n,k}|\ \leq \ \frac{(n-1)^{2(n-k)}}{2^{n-k}(n-k)!} \sum_{l=1}^{n-k}\sum_{\substack{ a_1+ \cdots +a_{l}=n-k }}\binom{n-k}{a_1,...,a_l}
\ = $
$$\frac{(n-1)^{2(n-k)}}{2^{n-k}(n-k)!} \sum_{l=1}^{n-k}l^{n-k}\ = \ \frac{(n-1)^{2(n-k)}S_{n-k}(n-k)}{2^{n-k}(n-k)!},$$
where $\ S_k(n)=1^k+ \cdots+ n^k.$

\end{proof}

\begin{rem}{\em  The function $\ \displaystyle S(x,y)=\int_{0}^{x}t^ydt=\frac{x^{y+1}}{y+1}\ $ arises as a continuous analogue of the $k$-power sum
$\ \ \displaystyle S_k(n)=\sum_{l=0}^{n}l^k=\frac{n^{k+1}}{k+1}+\frac{1}{k+1}\sum_{i=1}^{k}B_in^{k+1-i} \ \ $
for $\ k\in \mathbb{N}_{\geq 1}, \ $  where  the Bernoulli numbers $\ B_i \in \mathbb{Q}\ $ are defined \cite{aar} via the identity
$\ \displaystyle \frac{t}{1-e^{-t}} = \sum_{n=0}^{\infty}B_n\frac{t^n}{n!}.\ $   The difference between $\ S_k(n) \ $ and  $\ \displaystyle S(n,k)=\frac{n^{k+1}}{k+1}\  $  is a polynomial of degree $\ k, \ $ so $$\ \lim_{n\to \infty}\frac{\sum_{k=0}^{n}n^k}{\int_{0}^{n}t^kdt}  \ = \
\lim_{n\to \infty}\frac{S_k(n)}{S(n,k)} \ = \
\lim_{n\to \infty}\frac{\frac{n^{k+1}}{k+1}+\frac{1}{k+1}\sum_{i=1}^{r}B_in^{k+1-i} }{\frac{n^{k+1}}{k+1}}\ = \ 1.$$

}
\end{rem}
\noindent For $\  n\in \mathbb{N}\ $ set $\  \displaystyle e_{n}(x)=\sum_{k=0}^{n}\dfrac{x^k}{k!}.\ $

\begin{thm}\label{r4}{\em For each $\ n\in \mathbb{N},\ $
the function $\ \widetilde{r}(x,y,n) : \mathbb{R} \times\mathbb{ R } \rightarrow \mathbb{R}\ $ is continuous. We have $ \ \widetilde{r}(x,y,0)=1, \ $ $ \ \widetilde{r}(x,y,1)=x; \ $ for $\ n> 0 \ $ we have $\ \widetilde{r}(x,0,n)=x^n \ $ and  $\  \widetilde{r}(0,y,n)=0; \ $  for $ \ n> 0,\ x \neq 0  \ $ we have  $$\widetilde{r}(x,y,n)= x^ne_{n-1}\bigg(\frac{y(n-1)^2}{2x}\bigg).$$
}
\end{thm}

\begin{proof} $ \ \widetilde{r}(x,y,n) \ $ is an homogeneous polynomial  of degree $\ n, \ $ thus continuous. \\
By (\ref{e1})  we get
$ \ \widetilde{r}(x,y,0)= \widetilde{r}_{0,0}=1, \ \  \ \widetilde{r}(x,y,1)= \widetilde{r}_{1,1}x=x; \ \ $ for $\ n>0, \ $
we get $ \ \widetilde{r}(0,y,n)=0 \ $ and $ \ \widetilde{r}(x,0,n)=a_{n-1,0}x^n=x^n. \ $  For $\ n>0,\ x\neq 0, \ $ we have
$$\widetilde{r}(x,y,n)\ = \ \sum_{k=0}^{n} \widetilde{r}_{n,k}\ x^{k}y^{n-k} \ = \
\sum_{k=1}^{n} \dfrac{(n-1)^{2(n-k)}}{2^{n-k}(n-k)!}\ x^{k}y^{n-k}\ = $$
$$x^n\sum_{k=1}^{n} \frac{\big(\frac{y(n-1)^2}{2x}\big)^{n-k}}{(n-k)!}\ = \
x^ne_{n-1}\bigg(\frac{y(n-1)^2}{2x} \bigg) .$$
\end{proof}

\begin{examp}\label{1812b}{\em We list some properties of $\ \widetilde{r}(x,y,n).\ $   Let $\ n\in \mathbb{N}_{>0}\ $ and $\ a\neq 0.$\\
a) $\widetilde{r}(1,1,n)=e_{n-1}\big(\frac{(n-1)^2}{2} \big). \ \ \ \ \ \ \ \ \ \ \ \ \ \
\ \ \ \ \ \ \ \ \  b) \ \widetilde{r}(1,2,n)=e_{n-1}\big((n-1)^2\big). $ \\
c) $\widetilde{r}(x,1,n)=x^ne_{n-1}\big(\frac{(n-1)^2}{2x}\big). \ \ \ \ \ \ \ \ \ \
\ \ \ \ \  \ \ \ \ \
d)\ \widetilde{r}(x,-1,n)=x^ne_{n-1}\big(-\frac{(n-1)^2}{2x} \big). $\\
e) $\widetilde{r}((n-1)^2,2y,n)=(n-1)^{2n}e_{n-1}(y). \ \ \ \ \ \ \
f) \ \widetilde{r}(x,2(n-1)^{-2},n)=x^{n}e_{n-1}(x^{-1}).$ \\
g) $\widetilde{r}(ax,ay,n)=a^n\widetilde{r}(x,y,n). \ \ \ \ \ \ \ \ \ \ \ \
\ \ \ \ \ \ \ \ \
h)\ \widetilde{r}(x,y,n)=y^n\widetilde{r}(\frac{x}{y},1,n).$
}
\end{examp}

\begin{cor}\label{80124d}
{\em From properties \ e) \  and \ f) \ of \ Example \ref{1812b} \  we get:
  $$\lim_{n\to \infty}\frac{\widetilde{r}((n-1)^2,2y,n)}{(n-1)^{2n}} = e^y \ \ \ \ \ \text{and}
 \ \ \ \ \ \lim_{n\to \infty}x^{n}\ \widetilde{r}(x^{-1},2(n-1)^{-2},n) = e^{x}.
  $$
  }
\end{cor}

The Pochhammer symbol satisfies
$\ \displaystyle r(x,y,n+1)=(x+ny)r(x,y,n), $ \\
$$\ \frac{\partial r}{\partial x}(x,y,n)=
\sum_{i=0}^{n-1}r(x,y,i)r(x+(i+1)y,y,n-i-1), $$
$$\ \frac{\partial r}{\partial y}(x,y,n)=
\sum_{i=1}^{n-1}ir(x,y,i)r(x+(i+1)y,y,n-i-1).$$

\begin{prop}\label{1812a}{\em \
\begin{description}
  \item[a)] For $\ n\in \mathbb{N}_{\geq 2} \ $ the function $\ \widetilde{r}(x,y,n) \ $ satisfies the recursion
$$\widetilde{r}(x,y,n+1)\ = \ x\ \widetilde{r}\big(x,\frac{n^2}{(n-1)^2}y,n\big)\ + \ \frac{n^{2n}xy^n}{2^nn!}. \ \ \ \ \ \ \ \ \ \ \ \ \ \ \ $$
  \item[b)] For $\ n\in \mathbb{N}_{\geq 3} \ $ we have that
  $$\displaystyle x\frac{ \partial \widetilde{r}}{\partial x}(x,y,n) \ = \
  n\widetilde{r}(x,y,n) \ - \   \frac{(n-1)^2y}{2}\widetilde{r}\big(x,\frac{(n-1)^2}{(n-2)^2}y,n-1\big). $$
  \item[c)]  For $\ n\in \mathbb{N}_{\geq 3} \ $ we have that
  $$\displaystyle \frac{ \partial \widetilde{r}}{\partial y}(x,y,n) \ = \
 \frac{(n-1)^2}{2}\widetilde{r}\big(x,\frac{(n-1)^2}{(n-2)^2}y,n-1\big).
 \ \ \ \ \ \ \ \ \ \ \ \ \ \ \ \ \ \ \ \ $$
\end{description}

}
\end{prop}

\begin{proof}
Since $\ e_{n}(x)=e_{n-1}(x)+\frac{x^{n}}{n!}, \ $ item a)  follows from the identity
$$x^{n+1}e_n\big(\frac{yn^2}{2x}\big) \ = \ x^{n+1}e_{n-1}\big(\frac{yn^2}{2x}\big) \ + \
\frac{x^{n+1}}{n!}\bigg(\frac{yn^2}{2x}\bigg)^n .$$
Items b) and c) follow from the rules for derivatives and the identity
$\ \displaystyle \frac{de_{n+1}(x)}{dx}=e_n(x) .$
\end{proof}

For $\ n\in \mathbb{N}\ $ set $\  \displaystyle  2\mathrm{cosh}_{2n}(x)= \sum_{k=0}^{n}\frac{x^{2k}}{(2k)!}.$

\begin{thm}{\em Let $\ y \in \mathbb{R}_{>0}, \ n \in \mathbb{N}_{\geq 1} \ $ and  $\ X \ $ be a random variable with standard normal distribution. The expected valued of
$\ 2\mathrm{cosh}_{2(n-1)}((n-1)\sqrt{y}X)\ $ is equal to $\ \widetilde{r}(1,y,n): $
$$\widetilde{r}(1,y,n)\ = \ \frac{1}{\sqrt{2\pi}}\int_{-\infty}^{\infty}
2\mathrm{cosh}_{2(n-1)}((n-1)\sqrt{y}t) e^{-\frac{t^2}{2}}dt.$$
}
\end{thm}

\begin{proof}We have that
$$ \frac{1}{\sqrt{2\pi}}\int_{-\infty}^{\infty}
2\mathrm{cosh}_{2(n-1)}((n-1)\sqrt{y}t) e^{-\frac{t^2}{2}}dt \ = \
\frac{1}{\sqrt{2\pi}}\int_{-\infty}^{\infty}\bigg(\sum_{l=0}^{n-1}
\frac{((n-1)\sqrt{y}t)^{2l}}{(2l)!}\bigg) e^{-\frac{t^2}{2}}dt \ = $$
$$ \frac{1}{\sqrt{2\pi}}\int_{-\infty}^{\infty}\bigg(\sum_{k=1}^{n}
\frac{((n-1)\sqrt{y}t)^{2n-2k}}{(2n-2k)!}\bigg) e^{-\frac{t^2}{2}}dt \ = \ $$
$$\sum_{k=1}^{n}
\frac{((n-1)\sqrt{y})^{2n-2k}}{(2n-2k)!}\frac{1}{\sqrt{2\pi}}\int_{-\infty}^{\infty} t^{2n-2k}e^{-\frac{t^2}{2}}dt  \ = \ $$
$$ \sum_{k=1}^{n}
\frac{((n-1)\sqrt{y})^{2n-2k}}{(2n-2k)!}(2n-2k-1)!!\ = \
\sum_{k=1}^{n}
\frac{((n-1)\sqrt{y})^{2n-2k}}{2^{n-k}(n-k)!} \ =   $$
$$\sum_{k=1}^{n}
\frac{((n-1)\sqrt{y})^{2n-2k}}{2^{n-k}(n-k)!} \ = \  \sum_{k=1}^{n} \frac{\big(\frac{(n-1)^2y}{2}\big)^{n-k}}{(n-k)!} \ = \ \widetilde{r}(1,y,n).$$

\end{proof}

Next we introduce an extension of  $\ \widetilde{r}(x,y,n):\mathbb{R}_{x>0} \times\mathbb{ R }_{\geq 0}\times \mathbb{N}_{\geq 1}  \rightarrow \mathbb{R}\ $ to a map
$\ \widetilde{r}(x,y,z): \mathbb{R}_{>0}\times\mathbb{ R }_{\geq 0}\times \mathbb{R}_{> 0} \ $
defined as follows:
\begin{equation}\label{60124a}
\ \widetilde{r}(x,y,z) \ = \  x^ze_{z-1}\bigg(\frac{y(z-1)^2}{2x}\bigg) \ = \
x^{z}e^{\frac{y(z-1)^2}{2x}}\frac{\Gamma\big(z,\frac{y(z-1)^2}{2x}\big)}{\Gamma(z)}
\end{equation}
where for $\ z > 0 \ $ we set $\ \displaystyle e_{z-1}(x)=e^x\frac{\Gamma(z,x) }{\Gamma(z) }, \ $ and
$$\Gamma(z)  = \int_{0}^{\infty}t^{z-1}e^{-t}dt, \ \ \ \
\Gamma(z,x)  =  \int_{x}^{\infty}t^{z-1}e^{-t}dt, \ \ \ \ \gamma(z,x)  =  \int_{0}^{x}t^{z-1}e^{-t}dt. $$
The fact that formula (\ref{60124a}) gives an extension follows from the identities
$$ e_{n-1}(x)\ = \ \sum_{k=0}^{n-1} \frac{x^{k}}{k!} \ = \ e^x\frac{\Gamma(n,x) }{\Gamma(n) }
\ \ \ \ \ \text{for} \ \ \ n \geq 1$$
which can be shown by iterated integration by parts.

\begin{thm}\label{2012a}{\em The function $\ \widetilde{r}(x,y,n):\mathbb{R}_{x>0}
\times\mathbb{ R }_{\geq 0}\times \mathbb{N}_{\geq 1} \rightarrow \mathbb{R} \ $ admits a smooth extension
$\ \widetilde{r}(x,y,z): \mathbb{R}_{>0} \times\mathbb{ R }_{\geq 0}\times \mathbb{R}_{>0} \rightarrow \mathbb{R} \ $
given by the following formulae
$$\widetilde{r}(x,y,z) \ = \ x^ze^{\frac{y(z-1)^2}{2x}}\bigg(1-\frac{1}{\Gamma\left(z\right)}
\sum_{k=0}^{\infty}\frac{(-1)^{k}y^{k}(z-1)^{2k}}{2^k k!x^k(z+k) }\bigg),  $$
$$\widetilde{r}(x,y,z)  \ = \ x^ze^{\frac{(z-1)^2y}{2x}}-\frac{y^z(z-1)^{2z}}{2^z}\sum_{k=0}^{\infty}
\frac{y^k(z-1)^{2k}}{2^kx^k\Gamma\left(z+k+1\right)} .  $$
}
\end{thm}

\begin{proof}  The function $\ \widetilde{r}(x,y,z)\ $ given by (\ref{60124a}) is smooth on
$\ \mathbb{R}_{x>0} \times\mathbb{ R }_{\geq 0}\times \mathbb{R}_{>0}. \ $ The desired identities follows from the following formulae \cite{ni}:
\[\frac{\Gamma\left(z,x\right)}{\Gamma\left(z\right)}\ = \ \frac{ \Gamma(z)- \int_{0}^{x}t^{z-1}e^{-t}dt}{\Gamma(z)} \ = \  1-\frac{x^{z}}{\Gamma\left(z\right)}\sum_{k=0}^{\infty}\frac{(-1)^{k}
}{(z+k)}\frac{x^{k}}{k!} \ = \ 1-x^{z}e^{-x}\sum_{k=0}^{\infty}%
\frac{x^{k}}{\Gamma\left(z+k+1\right)},\]
\noindent $\text{and} \ \ \ \displaystyle e^x\frac{\Gamma\left(z,x\right)}{\Gamma\left(z\right)} \ = \  e^x\bigg(1-\frac{x^{z}}{\Gamma\left(z\right)}\sum_{k=0}^{\infty}\frac{(-1)^{k}
}{(z+k)}\frac{x^{k}}{k!}\bigg) \ = \ e^x-x^{z}\sum_{k=0}^{\infty}
\frac{x^{k}}{\Gamma\left(z+k+1\right)},$ \\
by first replacing $\ x \ $ by $\ \frac{y(z-1)^2}{2x}, \ $  and then multiplying by
$\ x^z.\ $

\end{proof}

\begin{examp}\label{1812c}{\em We list some properties of $\ \widetilde{r}(x,y,z).\ $ \\
a) $\widetilde{r}(1,1,z)=e^{\frac{(z-1)^2}{2}}\big(1-\frac{1}{\Gamma\left(z\right)}
\sum_{k=0}^{\infty}\frac{(-1)^{k}(z-1)^{2k}}{2^kk!(z+k)}\big). $ \\
b) $\widetilde{r}(1,1,z)=e^{\frac{(z-1)^2}{2}}-(z-1)^{2z}\sum_{k=0}^{\infty}
\frac{(z-1)^{2k}}{2^{z+k}\Gamma\left(z+k+1\right)}.$ \\
c) $ \widetilde{r}(x,2x(z-1)^{-2},z)=ex^{z}-x^{z}\sum_{k=0}^{\infty}
\frac{1}{\Gamma\left(z+k+1\right)}.$ \\
d) $\widetilde{r}(x,y,z)=y^z\widetilde{r}(\frac{x}{y},1,z).$\\
e) $\widetilde{r}(x,y,z) \leq  x^{z}e^{\frac{y(z-1)^2}{2x}}\ $ since
$\ \Gamma(z,x)\leq \Gamma(z).\ $
}
\end{examp}

We use the symbol $\ \sim \ $ to denote asymptotic equivalence. Stirling's formula yields $\  \displaystyle\Gamma(x)\sim  \sqrt{2\pi}e^{-x}x^{x-\frac{1}{2}} \ \ $  as $\ \ x \to \infty.  \ $ Therefore
$$\ \ \  r(x,y,z) \ \sim \ \frac{\sqrt{2\pi}}{\Gamma(\frac{x}{y})}y^ze^{-z}z^{z+\frac{x}{y}-\frac{1}{2}} \ \ \ \  \text{as} \ \ \  \ z \to \infty, \ \ \ \text{since by (1)}$$
$$r(x,y,z)\ = \ y^zr(\frac{x}{y},1,z)\ = \ y^z\frac{\Gamma(\frac{x}{y} + z)}{\Gamma(\frac{x}{y})}\ \sim $$
$$ \frac{\sqrt{2\pi}}{\Gamma(\frac{x}{y})}y^ze^{-z}e^{-\frac{x}{y}}
\big(\frac{x}{y}+z\big)^{z+\frac{x}{y}-\frac{1}{2}}\ \sim \ \frac{\sqrt{2\pi}}{\Gamma(\frac{x}{y})}y^ze^{-z}z^{z+\frac{x}{y}-\frac{1}{2}}.$$

A complete analysis of the asymptotic behaviour of $\ \widetilde{r}(x,y,z) \ $ is left open, however we study some interesting cases.

\begin{thm}\label{30124b}{\em \
\begin{description}
\item[a)]  $\displaystyle \widetilde{r}((z-1)^2,2y,z)\ \sim \ (z-1)^{2z}(e^y-\frac{y^z}{\Gamma(z+1)})\ \ $ as $\ \ z \to \infty.$
\item[b)] $\displaystyle \widetilde{r}((z-1)^2,2y,z)\ \sim \ \frac{y^{z-1}(z-1)^{2z}}{\Gamma(z)}\ \ $ as $\ \ y \to \infty.$
\item[c)] $\displaystyle x^z\widetilde{r}(\frac{1}{x},\frac{2}{(z-1)^2},z)\ \sim \  e^x-\frac{x^z}{\Gamma(z+1)}\ \ $ as $\ \ z \to \infty.$
\item[d)] $\displaystyle  x\widetilde{r}(\frac{1}{x},\frac{2}{(z-1)^2},z) \ \sim \ \frac{1}{\Gamma(z)}\ \ $ as $\ \ x \to \infty.$
\item[e)] $\displaystyle \widetilde{r}(1,\frac{2z}{(z-1)^2},z)\ \sim \ \frac{1}{2}e^z \ \ $ as $\ \ z \to \infty.$
\end{description}

}
\end{thm}

\begin{proof}
We use three known facts \cite{n, ni, ol} on the asymptotic behaviour of the regularized incomplete gamma function:
 $$\ \frac{\Gamma(z,y)}{\Gamma(z)}\ \sim \  1-\frac{y^ze^{-y}}{\Gamma(z+1)}  \ \ \ \mathrm{as} \ \ \ z \to \infty ; \ \ \ \ \ \ \ \ \ \ \ \ \ \ \ \ \ \ \ \ \ \ \ \ \ \ \ \ \ \ \ \ \ \ \ \ \ \ \ \ \ \ \ \ \ \ \ \ \    $$
 $$\ \frac{\Gamma(z,y)}{\Gamma(z)}\ \sim \ \frac{y^{z-1}e^{-y}}{\Gamma(z)} \ \ \ \mathrm{as} \ \ \ y \to \infty ; \ \ \ \ \ \ \ \ \ \ \ \ \ \ \ \ \ \ \ \ \ \ \ \ \ \ \ \
  \ \ \ \ \ \ \ \ \ \ \ \ \ \ \ \ \ \ \ \ \ \ \ \ \ \ \ \ $$
 $$\ \frac{\Gamma(z,z)}{\Gamma(z)}\ \sim \ \frac{\sqrt{\frac{\pi}{2}}e^{-z}z^{z-\frac{1}{2}} }{\sqrt{2\pi}e^{-z}z^{z-\frac{1}{2}}}
 \ \sim \ \frac{1}{2} \ \ \
 \mathrm{as} \ \ \ z \to \infty.  \  \ \ \ \ \ \ \ \ \ \ \ \ \ \ \ \ \ \ \ \ \ \ \ \ \ \ \ \ \ \ \ \ \ \ \ \ \ \ \ \ \ $$

\noindent We obtain the desire results as follows.\\

\noindent $a) \displaystyle \ \ \ \widetilde{r}((z-1)^2,2y,z)\ = \ (z-1)^{2z}e^{y}\frac{\Gamma(z,y)}{\Gamma(z)} \ \sim \
(z-1)^{2z}\big( e^y-\frac{y^z}{\Gamma(z+1)} \big).$\\

\noindent $b)\displaystyle \ \ \ \widetilde{r}((z-1)^2,2y,z)\ = \ (z-1)^{2z}e^{y}\frac{\Gamma(z,y)}{\Gamma(z)} \ \sim \
\frac{y^{z-1}(z-1)^{2z}}{\Gamma(z)}.  \ \ \   \ \ \   \ \ \  \ \ \  \ \ \   \ \ \  $
\noindent $c)\displaystyle \ \ \ \widetilde{r}(\frac{1}{x},\frac{2}{(z-1)^2},z)\ = \ x^{-z}e^{x}\frac{\Gamma(z,x)}{\Gamma(z)} \ \sim \
x^{-z}\big( e^x-\frac{x^z}{\Gamma(z+1)} \big). $ \\

\noindent $d)\displaystyle \ \ \ \widetilde{r}(\frac{1}{x},\frac{2}{(z-1)^2},z)\ = \ x^{-z}e^{x}\frac{\Gamma(z,x)}{\Gamma(z)} \ \sim \
\frac{x^{-z}x^{z-1}}{\Gamma(z)}\ = \ \frac{1}{x\Gamma(z)}. \ \  \ \ \ \   \ \ \  \ \ \  \ \ \   \  \ \    $ \\

\noindent $e)\displaystyle \ \ \ \widetilde{r}(1,\frac{2z}{(z-1)^2},z)\ = \ e^{z}\frac{\Gamma(z,z)}{\Gamma(z)} \ \sim \ \frac{1}{2}e^z.     \ \ \   \ \ \  \ \ \  \ \ \   \ \   \ \ \   \ \  $
\end{proof}

\section{Continuous Analogue - Second Approach}

Note that formula (\ref{e1}) may be written for $\ n\geq 1 \ $ as follows
$$\widetilde{r}(x,y,n)\ = \ \sum_{k\in \mathbb{Z}\cap [1,n]} \dfrac{(n-1)^{2(n-k)}}{2^{n-k}\Gamma(n-k+1)}\ x^{k}y^{n-k}, $$
so $\ \widetilde{r}(x,y,n)\ $ is a weighted sum over the lattice points in the convex polytope $[1,n] \subseteq \mathbb{R}, $ with weights given by evaluation of the continuous function $$\frac{(z-1)^{2(z-s)}}{2^{z-s}\Gamma(z-s+1)}x^sy^{z-s} \ \ \ \ \text{for} \ \ \ \  1\leq s\leq z. $$
Thus according to the methodology outlined in the introduction, a continuous analogue for $\ r(x,y,n)\ $ may be defined as follows:
$$\rho(x,y,z)\ = \ \int_{1}^z
\dfrac{(z-1)^{2(z-s)}}{2^{z-s}\Gamma(z-s+1)}\ x^{s}y^{z-s}\ ds.$$
Note that $\ \rho(x,y,1)=0\ $ while $\ r(x,y,1)=x; \ $ this discrepancy is expected as we are trading a sum over a point by an integral over an interval of length $ 0. \ $ Setting $\ t=z-s \ $ in the above formula  we get
\begin{equation}\label{2012b}
\rho(x,y,z)\ = \
x^z\int_{0}^{z-1}\frac{\big(\frac{y(z-1)^2}{2x}\big)^t}{\Gamma(t+1)}\ dt
\ = \ x^z\mathrm{E}(\frac{y(z-1)^2}{2x}, z-1) \ ,
\end{equation}
$$ \text{where}\ \ \ \ \mathrm{E}(x,z) \ = \ \int_{0}^{z}\frac{x^t}{\Gamma(t+1)}\ dt \ \ \ \
\text{for} \ \ \ \ x\geq 0, \ z\geq 0. $$

\begin{rem}{\em Since $\ \displaystyle e_{n}(x) =  \sum_{k=0}^{n} \frac{x^{k}}{k!}=\sum_{k\in \mathbb{Z}\cap [0,n]}\frac{x^k}{\Gamma(k+1)}, \ $ we see that
$\ \mathrm{E}(x,z) \ $ arises as a continuous analogue of $\ \displaystyle e_{n}(x).  \ $ Letting $\ n\to \infty \ $ in this analogy  the exponential function $\ \displaystyle e^x \ $ is replaced by $ \ \displaystyle \nu(x) =\int_{0}^{\infty}\frac{x^t}{\Gamma(t+1)}\ dt.$}
\end{rem}

\begin{rem}{\em   The geometric sum polynomial
 $ \ \displaystyle S_m(x)=\sum_{k=0}^{m}x^k=\frac{x^{m+1}-1}{x-1}, \ $ for $\ x>1,\ $ admits a natural extension $\ \displaystyle S_y(x)=\frac{x^{y+1}-1}{x-1}\ $ for $\ y\in \mathbb{R}. \ $ Furthermore the function $\ \displaystyle G(x,y)= \int_{0}^{y}x^tdt=\frac{x^y-1}{\mathrm{ln}(x)}\ $
 arises as a continuous analogue of   $\ S_m(x).\ $
For $\ y>0 \ $ it turns out that $\ G(x,y)\  $ is negligible in comparison to  $\ S_y(x) \ $ as $\ x\to \infty, \ $ indeed
$$\lim_{x\to \infty}\frac{G(x,y)}{S_y(x)} \ = \ \lim_{x\to \infty}\frac{\int_{0}^{y}x^tdt }{\frac{x^{y+1}-1}{x-1}}  \ = \
\lim_{x\to \infty}\frac{(x^y-1)(x-1)}{(x^{y+1}-1)\mathrm{ln}(x)} \ = \ 0.$$
}
\end{rem}

To understand  $\ \rho(x,y,z) \ $ we study the function
$$\mathrm{E}(x,z)\ = \  \int_{0}^{\infty}\frac{x^t}{\Gamma(t+1)}\ dt \ - \
\int_{z}^{\infty}\frac{x^t}{\Gamma(t+1)}\ dt \ = \ \nu(x) - \nu(z,x).$$ The functions
$ \ \ \displaystyle \nu(x)=\int_{0}^{\infty}\frac{x^t}{\Gamma(t+1)} dt\ \ $ and
$ \ \ \displaystyle \nu(x,z)= \int_{z}^{\infty}\frac{x^t}{\Gamma(t+1)} dt\ \ $ have a quite long history, see \cite{ba,g} and the references therein; \ they are often studied as instances of the so called $ \ \mu $-function of three variables given by
$$  \mu(x,\beta, \alpha)\ = \ \int_0^{\infty}\frac{x^{\alpha+t}\ t^{\beta}}{\Gamma(\alpha +t+1)\Gamma(\beta+1) }\ dt  .$$ Indeed  $ \ \nu(x)=\mu(x,0,0) \ $ and $\  \nu(x,z)=\mu(x,0,z). \  $ Note that $$ \displaystyle \ \nu(e^{-s})\ = \ \int_{0}^{\infty}\frac{e^{-st}}{\Gamma(t+1)} dt\ $$
is the Laplace transform of $\ \displaystyle \frac{1}{\Gamma(t+1)};\ $  similarly  $\ \mathrm{E}(e^{-s},z)\ $ is the cut-off of the Laplace transform of $\ \displaystyle \frac{1}{\Gamma(t+1)}\ $ at level $\ z.$

The reciprocal gamma  function is entire. Wrench in \cite{w} provides a representation of it as an absolutely convergent power series:
\begin{equation}\label{80124a}
\ \frac{1}{\Gamma(z)} \ = \
z\ \mathrm{exp}\big({\gamma z - \sum_{k=2}^{\infty}(-1)^k\frac{\zeta(k)}{k}z^k}\big)\ = \
 z\ \mathrm{exp}\big({\sum_{k=1}^{\infty}(-1)^{k+1}\frac{\widehat{\zeta}(k)}{k}z^k}\big)
\end{equation}
$$\text{where }  \ \ \   \gamma \ = \ \lim _{n\rightarrow \infty }\big(\sum _{k=1}^{n}{\frac {1}{k}}-\ln(n)\big) \ = \ \int _{1}^{\infty }\left({1 \over \lfloor x\rfloor }-{1 \over x}\right) dx \ \approx \
0.5772 \ $$ is the Euler-Mascheroni constant;  $\ \ \displaystyle \zeta(k)= \sum_{n=1}^{\infty}\frac{1}{n^k} \ $ are the integer values of the Riemann zeta function for $\ k \geq 2;\ \ $ and finally $\ \widehat{\zeta}(1)=\gamma \ $ and $\ \widehat{\zeta}(k) = \zeta(k)\ $ for $\ k\in \mathbb{N}_{\geq 2}.$

\begin{rem}{\em According to the terminology used in this work the  logarithmic function $\ \displaystyle \mathrm{ln}(x)=\int_{1}^{x}\frac{1}{t}dt\ $ arises as a continuous analogue of the harmonic sum $\ \displaystyle H_n=\sum_{k=1}^{n}\frac{1}{k}.\ $ The difference between $\ H_n \ $ and $\ \mathrm{ln}(n)\ $  approaches  $\ \gamma\ $ as $\ n \to \infty. \ $
}
\end{rem}

\begin{prop}\label{2812a}{\em The entire function $ \  \displaystyle \frac{1}{\Gamma(t+1)} \ $ may be written as a absolutely convergent power series as follows:
\begin{equation}\label{80124b}
  \displaystyle \frac{1}{\Gamma(t+1)}  \ = \
  \mathrm{exp}\big({\sum_{k=1}^{\infty}(-1)^{k+1}\frac{\widehat{\zeta}(k)}{k}t^k}\big)\ = \ \sum_{n=0}^{\infty}c_nt^n.
\end{equation}
The following properties hold.
\begin{description}
  \item[a)] $ c_0=1\ $ and for $\ n\in \mathbb{N}\ $ we have that
  $\ \ \displaystyle  (n+1)c_{n+1}  =  \sum_{k=0}^{n}(-1)^{k}\widehat{\zeta}(k+1)c_{n-k}.$
  \item[b)]  For $\ n\in \mathbb{N}_{\geq 1} \ $ we have that
  $$c_n  \ = \  \sum_{\substack{ k_1+ \cdots +k_{l}=n \\
k_i\geq 1, \ l\geq 1 }} \frac{(-1)^{n+l}}{l!}\frac{\widehat{\zeta}(k_1)}{k_1} \cdots \frac{\widehat{\zeta}(k_l)}{k_l} .$$
  \item[c)] $\displaystyle \frac{x^t}{\Gamma(t+1)}\ = \ \mathrm{exp}\big({(\mathrm{ln}(x)+\gamma) t + \sum_{k=2}^{\infty}(-1)^{k+1}\frac{\zeta(k)}{k}t^k}\big).$
  \item[d)] $ \displaystyle \frac{x^t}{\Gamma(t+1)}= e^{\mathrm{ln}(x)t}\sum_{n=0}^{\infty}c_nt^{n}=
      \sum_{n=0}^{\infty}c_n(x)t^{n} \ \   $ where
  $  \ \ \displaystyle c_n(x) =
  \sum_{k=0}^{n}c_{n-k}\frac{\mathrm{ln}^{k}(x)}{k!}. $
   \item[e)] $  \displaystyle \frac{d^k}{dt^k}\frac{x^t}{\Gamma(t+1)} = \sum_{n=0}^{\infty}(n+1)_{\overline{k}}c_{n+k}(x)t^n \ \ \ $ and
    $ \ \ \  \displaystyle \int^{[k]}\frac{x^t}{\Gamma(t+1)}dt^{[k]} = \sum_{n=k}^{\infty}\frac{c_{n-k}(x)}{(n)_{\underline{k}}}t^n. \ \ $
  \item[f)]  $ \displaystyle E(x,z)=\int_{0}^{z} \frac{x^t}{\Gamma(t+1)}dt = \sum_{n=1}^{\infty}c_{n-1}(x)\frac{z^{n}}{n} \ \ $ is smooth
     on $\ \mathbb{R}_{>0}\times \mathbb{R}_{>0}.$
 \item[g)] $ \displaystyle \frac{\partial^k E}{\partial z^k}(x,z) = \sum_{n=0}^{\infty}(n+1)_{\overline{k-1}}c_{n+k-1}(x)z^{n} \ \ $ for $\ \ k\in \mathbb{N}_{\geq 1}.$
 \item[h)] $ \displaystyle \frac{d c_n(x)}{d x} = \frac{c_{n-1}(x)}{x}\ $ for
 $\ n\in \mathbb{N}_{\geq 1},\ $ and $\ \displaystyle  \frac{d c_0(x)}{d x}=0. \ $ Furthermore $\ c_n(1)=c_n.$
 \item[i)] $ \displaystyle \big(x\frac{\partial}{\partial x}\big)^kE(x,z)= \sum_{n=k+1}^{\infty} c_{n-k-1}(x)\frac{z^{n}}{n} \ \ $ for $\ \ k\in \mathbb{N}_{\geq 1}.$
 \item[j)]   $ \ \displaystyle  E(e^{\gamma},z)  \leq  z\ \ $  and  $ \ \ \displaystyle  E(x,z)  \leq  \frac{ x^ze^{-\gamma z}-1}{\mathrm{ln}(x)-\gamma}\ \ $ for $ \displaystyle \ x \neq e^{\gamma}. \ $

\end{description}
 }
\end{prop}

\begin{proof}
Identity (\ref{80124b}) follows from the following facts: \  i)  $\displaystyle \frac{1}{\Gamma(t+1)} \ $ is an entire function as it is the composition of the entire functions $\ \displaystyle \frac{1}{\Gamma(t)} \ $ and $\ t+1;  $ \ ii) $ \ \displaystyle \mathrm{exp}\big({\sum_{k=1}^{\infty}(-1)^{k+1}\frac{\widehat{\zeta}(k)}{k}z^k}\big) \ $ is an entire function since it is a power series that converges to $\ \displaystyle \frac{1}{z\Gamma(z)} \ $ for $\ z\neq 0,\ $ and to $\ 1 \ $ for $\ z=0;\ $ iii) identity (\ref{80124b}) holds for $\ \mathrm{Re}(z)>0 \ $ by
(\ref{80124a}) and $\ \Gamma(z+1)=z\Gamma(z). \ $

Item a) follows from (\ref{80124b}) taking a derivative. Item b) follows from (\ref{80124b}) developing the exponential function as a power series. Items c) and d) follow from (\ref{80124b}) and the identity $\ x^t=e^{t\mathrm{ln}(x)}. \ $ Items \ e), f), g), h) \ follow from the rules for the derivatives of polynomials and power series. Item i) is obtained by iterated applications of item h). For item j) we use \cite{a, q} that $\ \displaystyle \Gamma(t+1) \geq  e^{\gamma t} \ $ for $\ t\geq 0.\ $
We have that $$ E(e^{\gamma},z) \ = \
\int_{0}^{z} \frac{e^{\gamma t}}{\Gamma(t+1)}dt \ \leq \ \int_{0}^{z} e^{\gamma t}e^{-\gamma t } dt\ = \ \int_{0}^{z}dt\ = \ z.$$
 If $\ x\neq e^{\gamma}, \ $ then
 $$ E(x,z) \ = \
\int_{0}^{z} \frac{e^{\mathrm{ln}(x)t}}{\Gamma(t+1)}dt \ \leq \ \int_{0}^{z} e^{(\mathrm{ln}(x)-\gamma )t} dt\ = \ \frac{x^ze^{-\gamma z}-1}{\mathrm{ln}(x)-\gamma}.$$
\end{proof}

 A recursive formula for the Taylor coefficients at $\ t=0\ $ of $ \ \displaystyle \frac{1}{\Gamma(t)} \ $   was introduced by Bourgubt \cite{b}. Sakata in \cite{s} gives a formula for these coefficients using  $\ \gamma \ $ and the multiple zeta values $\  \displaystyle \zeta(k_1,...,k_r)=
\sum_{0<n_1<\cdots <n_k} \frac{1}{n_1^{k_1}\cdots n_r^{k_r}} \  $ with $\ k_i\geq 2.$\\

\begin{thm}{\label{2812b}\em The following properties hold for $ \rho:\mathbb{R}_{>0}\times \mathbb{R}_{>0} \times \mathbb{R}_{> 1} \rightarrow \mathbb{R}  $ given by (\ref{2012b}).
\begin{description}
  \item[a)] $\displaystyle \rho(x,y,z) \ = \ x^zE(\frac{y(z-1)^2}{2x},z-1)\ = \  x^z\sum_{n=1}^{\infty}c_{n-1}\big(\frac{y(z-1)^2}{2x} \big)\frac{(z-1)^n}{n} $ \\
       is a smooth function.

  \item[b)] $\displaystyle y\frac{\partial \rho}{\partial y}\  = \
  x^{z}\sum_{n=2}^{\infty}c_{n-2}\big(\frac{y(z-1)^2}{2x} \big)\frac{(z-1)^{n}}{n}.$

  \item[c)]  $\displaystyle x\frac{\partial \rho}{\partial x} \ + \ y \frac{\partial \rho}{\partial y}\ = \ z\rho.$

  \item[d)] $\displaystyle
   \frac{\partial \rho}{\partial z} \ = \
  \mathrm{ln}(x)\rho \ +\ \frac{2y}{z-1} \frac{\partial \rho}{\partial y} \ + \ x^z\sum_{n=0}^{\infty}
  c_{n}\big(\frac{y(z-1)^2}{2x}\big)\ (z-1)^{n}.$

  \item[e)] $\displaystyle \rho(e^{-\gamma},\frac{2}{(z-1)^2},z) \ \leq
  (z-1)e^{-\gamma z}. $

  \item[f)] $\displaystyle \rho(x,\frac{2}{(z-1)^2},z) \ \leq \
  \frac{x^{z}-xe^{\gamma(1-z)}}{\mathrm{ln}(x)+\gamma}  \ \ \ $
  for $ \ \ \ \displaystyle x \neq e^{-\gamma}.$

\end{description}
}
\end{thm}

\begin{proof}
Item a) follows from  Proposition \ref{2812a}f. For items b), c), d) compute $\ \displaystyle \frac{\partial \rho}{\partial y}, \ \frac{\partial \rho}{\partial x}, \ \frac{\partial \rho}{\partial z} \ $ using Proposition \ref{2812a}f and Proposition \ref{2812a}h.
For items \ e),  f) \  apply  Proposition \ref{2812a}j.
\end{proof}

The function $\ \Gamma(t+1) \ $ is strictly convex on $\ \mathbb{R}_{\geq 0}. \ $  Since $\ \Gamma(0+1)=1=\Gamma(1+1), \ $ its unique minimum its achieved at a point $\ a\in (0,1). \ $ It is known \cite{dc} that $\ a\approx  0.4616 ,\ $ while
the minimum value of $\ \Gamma(t+1) \ $ is $\ m=\Gamma(a+1)\approx 0.8856.  \ $

\begin{prop}\label{20124}{\em  \
\begin{description}
  \item[a)] If $\ x\geq 1 \ $ and $\ \lfloor z\rfloor \geq 2,\ $  then
  $$\frac{1}{x}\big(e_{\lfloor z \rfloor }(x)-1\big) \ \leq \  E(x,z) \  \leq \ x\big(e_{\lceil z \rceil -1}(x)+\frac{1-m}{m}\big).$$
  \item[b)] If $\ x\geq 1, \ $  then
  $\ \ \  \displaystyle \frac{1}{x}\big(e^x-1\big)  \ \leq \  \nu(x) \  \leq \ x\big(e^x +\frac{1-m}{m}\big).$

  \item[c)] If $\ 0<x<1 \ $ and $\ \lfloor z\rfloor \geq 2,\ $   then $$   e_{\lfloor z\rfloor}(x)-1 \ \leq \  E(x,z) \ \leq \ e_{\lceil z\rceil-1}(x) + \frac{1-m}{m}.$$

  \item[d)] If $\ 0<x<1, \ $  then $\ \ \  \displaystyle  e^{x}-1 \ \leq \  \nu(x) \ \leq \ e^x+\frac{1-m}{m}.$
\end{description}
  }
\end{prop}

\begin{proof} Items b) and d) follow from a) and c), respectively, taking the limit
$\ z \to \infty. \ $
\noindent $\displaystyle a_{\leq}) \ \ \ E(x,z)\ \leq \ \int_{0}^{\lceil z\rceil}\frac{x^t}{\Gamma(t+1)}dt \ = \  \int_{0}^{1}\frac{x^t}{\Gamma(t+1)}dt \ + \ \sum_{i=1}^{\lceil z\rceil-1}\int_{i}^{i+1}\frac{x^t}{\Gamma(t+1)}dt \ \leq $
$$  \frac{x}{m} + \sum_{i=1}^{\lceil z\rceil-1}\frac{x^{i+1}}{\Gamma(i+1)}dt \ = \
 x\big(e_{\lceil z\rceil-1}(x) + \frac{1-m}{m}\big).\ \ \ \ \ \ \ $$
\noindent $\displaystyle a_{\geq}) \ \ \ E(x,z)\ \geq \ \int_{0}^{\lfloor z\rfloor}\frac{x^t}{\Gamma(t+1)}dt \ = \  \int_{0}^{1}\frac{x^t}{\Gamma(t+1)}dt\ + \ \sum_{i=1}^{\lfloor z\rfloor-1}\int_{i}^{i+1}\frac{x^t}{\Gamma(t+1)}dt \ \geq $
$$  1 + \sum_{i=1}^{\lfloor z\rfloor-1}\frac{x^{i}}{\Gamma(i+2)}dt \ = \
 1 + \frac{1}{x}\sum_{i=1}^{\lfloor z\rfloor-1}\frac{x^{i+1}}{\Gamma(i+2)}dt \ = \
 1 + \frac{1}{x}\big(e_{\lfloor z\rfloor}(x) -x-1\big). \ \ \ \    $$
\noindent $\displaystyle c_{\leq}) \ \ \   E(x,z)\ \leq \  \int_{0}^{\lceil z\rceil}\frac{x^t}{\Gamma(t+1)}dt \ = \ \int_{0}^{1}\frac{x^t}{\Gamma(t+1)}dt \ + \ \sum_{i=1}^{\lceil z\rceil-1}\int_{i}^{i+1}\frac{x^t}{\Gamma(t+1)}dt \ \leq \ $
$$ \frac{1}{m}\ + \ \sum_{i=1}^{\lceil z\rceil-1}\frac{x^{i}}{\Gamma(i+1)}\ = \
e_{\lceil z\rceil-1}(x) \ + \ \frac{1-m}{m}.   \ \ \ \ \ \ \ \ \ \ \ \ $$

\noindent $\displaystyle c_{\geq}) \ \ \ E(x,z)\ \geq \ \int_{0}^{\lfloor z\rfloor}\frac{x^t}{\Gamma(t+1)}dt \ = \   \int_{0}^{1}\frac{x^t}{\Gamma(t+1)}dt\ + \ \sum_{i=1}^{\lfloor z\rfloor-1}\int_{i}^{i+1}\frac{x^t}{\Gamma(t+1)}dt \ \geq \ $
$$  x + \sum_{i=1}^{\lfloor z\rfloor-1}\frac{x^{i+1}}{\Gamma(i+2)}dt \ = \
 e_{\lfloor z\rfloor}(x)-1.  \ \ \ \ \ \ \ \ \ \ \  \ \ \ \ \ \ \ \ \ \ \ \ \  $$
\end{proof}



Next result follows from Proposition \ref{20124}.

\begin{thm}\label{40124b}{\em  \
\begin{description}
\item[a)] If $\ y(n-1)^2\geq 2x \ $ and $\ n \geq 3, \ $  then
$$ \frac{2x}{y(n-1)^2}\big(\widetilde{r}(x,y,n)  \ - \ x^n\big) \ \leq \   \rho(x,y,n)
 \  \leq \  \frac{y(n-1)^2}{2x}\big(\widetilde{r}(x,y,n)\ + \ \frac{1-m}{m}x^{n}\big).$$
\item[b)] If $\ y\geq 1, \ $  then
  $\displaystyle \ \ \ \frac{1}{y}\big(e^{y}-1\big)
 \ \leq \ \lim_{n \to \infty}\frac{\rho((n-1)^2,2y,n)}{(n-1)^{2n}}  \ \leq \
 y\big(e^{y} +\frac{1-m}{m}\big). $

  \item[c)] If $\ y(n-1)^2<2x \ $ and $\ n\geq 3, \ $  then
  $$ \widetilde{r}(x,y,n) -x^n\ \leq \  \rho(x,y,n) \ \leq \ \widetilde{r}(x,y,n) +\frac{1-m}{m}x^n.$$

  \item[d)] If $\ y<1, \ $  then $\ \ \ \displaystyle    e^{y}-1 \ \leq \ \lim_{n \to \infty}\frac{\rho((n-1)^2,2y,n)}{(n-1)^{2n}}  \ \leq \ e^{y}+\frac{1-m}{m}. $
\end{description}
}
\end{thm}

\begin{proof}
Item a) follows from Proposition \ref{20124}a using $\ \rho(x,y,n) \ = \ x^nE(\frac{y(n-1)^2}{2x},n-1)\ $ and
$$e_{n-2}(\frac{y(n-1)^2}{2x})\ = \ e_{n-1}(\frac{y(n-1)^2}{2x})-
\frac{ y^{n-1}(n-1)^{2(n-1)} }{ 2^{n-1}x^{n-1}(n-1)! }\ \leq \ e_{n-1}(\frac{y(n-1)^2}{2x}).$$
Item b) follows from Proposition \ref{20124}b and the fact that
$$ \lim_{n \to \infty}\frac{\rho((n-1)^2,2y,n)}{(n-1)^{2n}}\ = \
\lim_{n \to \infty}\frac{(n-1)^{2n}E(y,n-1)}{(n-1)^{2n}}\ = \ \nu(y).$$
Item c) follows from Proposition \ref{20124}c. \ \
Item d)  follows from  Proposition \ref{20124}d.
\end{proof}

Our final result provides asymptotic bounds for   $\ \rho(x,y,n) \ $ in some interesting cases.

\begin{thm}\label{50124a}
\label{30124c}{\em Let $\ n \in \mathbb{N}_{\geq 3}.$\
\begin{description}
\item[a)]   If $\ (n-1)^2y\geq 2x\geq 2, \ $  then
 $\ \ \displaystyle x\rho(x,y,n) \ \leq \ n^2y\widetilde{r}(x,y,n).$
\item[b)] If $\ y \geq 1, \ $ then   $\ \ \displaystyle \rho((n-1)^2,2y,n)=  O(n^{2n})\ \ $ as $\ \ n \to \infty.$
\item[c)] $\displaystyle \rho(1,\frac{2n}{(n-1)^2},n) =  O(ne^n) \ \ $ as $\ \ n \to \infty.$
\item[d)] $\displaystyle \lim_{n\to \infty}\rho(1,\frac{2}{(n-1)^2},n)\ = \  \nu(1)\  = \ \int_{0}^{\infty}\frac{1}{\Gamma(t+1)} \ \approx \ 2.2665 \ \ \ $ as $\ \ \ n \to \infty.$
\end{description}
}
\end{thm}

\begin{proof}
For item a) note first that
$$\ \widetilde{r}(x,y,n)=  x^ne_{n-1}\big(\frac{y(n-1)^2}{2x}\big)  \geq  x^n \ \ \ \
\text{and} \ \ \ \ m \approx 0.8856 \geq\frac{1}{2}. \ $$ From Theorem \ref{40124b}a we get
$$ \frac{x}{y}\rho(x,y,n)\ \leq \ \frac{(n-1)^2}{2}\big(\widetilde{r}(x,y,n)\ + \ \frac{1-m}{m}x^{n}\big)
\ \leq $$
$$ \frac{(n-1)^2}{2}\big(\widetilde{r}(x,y,n)\ + \ \frac{1-m}{m}\widetilde{r}(x,y,n)\big) \ = \ $$
$$\frac{(n-1)^2}{2m}\widetilde{r}(x,y,n) \ \leq \ (n-1)^2\widetilde{r}(x,y,n)\ \leq \  n^2\widetilde{r}(x,y,n).$$
Items  b) and c)  are obtained from item a) and Theorem \ref{30124b} as follows:\\

\noindent b) $\ y \geq 1 \ $ implies the conditions of item a) hold, thus
$$ (n-1)^2\rho((n-1)^2,2y,n)\ \leq \ 2n^2y\widetilde{r}((n-1)^2,2y,n), \ \ \ \text{and by Theorem \ref{30124b}a}$$
$$\rho((n-1)^2,2y,n)\ = \ O(\frac{n^2}{(n-1)^2}(n-1)^{2n}2y(e^y-\frac{y^n}{n!})) \ = $$ $$O((n-1)^{2n})\ = \ O(n^{2n}).$$

\noindent c)  Conditions of item a) hold, so
$$ \rho(1,\frac{2n}{(n-1)^2},n)\ \leq \  \frac{2n^3}{(n-1)^2}\widetilde{r}(1,\frac{2n}{(n-1)^2},n),  $$
thus by Theorem \ref{30124b}e we get that
$$\rho(1,\frac{2n}{(n-1)^2},n)\ = \ O(n\widetilde{r}(1,\frac{2n}{(n-1)^2},n))\ = \ O(n\frac{e^n}{2})\ = \
O(ne^n).$$

\noindent d)  We have that $$ \lim_{n\to \infty}\rho(1,\frac{2}{(n-1)^2},n)\ = \
 \lim_{n\to \infty}E(\frac{2(n-1)^2}{2(n-1)^2},n-1) \ = \ $$
 $$ \lim_{n\to \infty}E(1,n-1) \ = \ \nu(1)\ = \  \int_{0}^{\infty}\frac{1}{\Gamma(t+1)} \ \approx \ 2.2665.$$
See \cite{oe} for the approximated value of $\ \nu(1).$
\end{proof}

\

\noindent ragadiaz@gmail.com \\
\noindent Escuela de Matem\'aticas, Universidad Nacional de Colombia - Sede Medell\'in, Colombia

\end{document}